\documentclass[oneside,english]{amsart}
\usepackage[T1]{fontenc}
\usepackage[latin9]{inputenc}
\usepackage{geometry}
\usepackage{amstext}
\usepackage{amsthm}
\usepackage{amssymb}

\makeatletter
\numberwithin{equation}{section}
\numberwithin{figure}{section}
 \theoremstyle{definition}
 \newtheorem*{defn*}{\protect\definitionname}
\theoremstyle{plain}
\newtheorem{thm}{\protect\theoremname}
  \theoremstyle{plain}
  \newtheorem{cor}[thm]{\protect\corollaryname}
  \theoremstyle{remark}
  \newtheorem{rem}[thm]{\protect\remarkname}
  \theoremstyle{plain}
  \newtheorem{lem}[thm]{\protect\lemmaname}
  \theoremstyle{plain}
  \newtheorem{prop}[thm]{\protect\propositionname}
  \theoremstyle{plain}
  \newtheorem{criterion}[thm]{\protect\criterionname}
  \theoremstyle{definition}
  \newtheorem{example}[thm]{\protect\examplename}

\usepackage{xypic, pstricks}
\usepackage{mathrsfs, setspace, fancyhdr, times, bm}
\DeclareFontFamily{OT1}{pzc}{}
\DeclareFontShape{OT1}{pzc}{m}{it}%
             {<-> s * [1.195] pzcmi7t}{}
\DeclareMathAlphabet{\mathscr}{OT1}{pzc}%
                                 {m}{it}
\usepackage[colorlinks=true,pagebackref=true]{hyperref} 
\hypersetup{backref}

\makeatother

\usepackage{babel}
  \providecommand{\corollaryname}{Corollary}
  \providecommand{\criterionname}{Criterion}
  \providecommand{\definitionname}{Definition}
  \providecommand{\examplename}{Example}
  \providecommand{\lemmaname}{Lemma}
  \providecommand{\propositionname}{Proposition}
  \providecommand{\remarkname}{Remark}
\providecommand{\theoremname}{Theorem}

\begin{document}
\psset{unit=0.8}
\newsavebox{\Xalone}
\savebox{\Xalone}{\pscurve(0,0)(1,0)(1.5,-0.2)(2,0)(2.5,0.2)(3,0)}

\newsavebox{\XaloneHole}
\savebox{\XaloneHole}
{\pscurve(0,0)(1,0)(1.2,-0.1)
\pscurve(1.7,-0.1)(2,0)(2.5,0.2)(3,0)
}

\newsavebox{\XtCzero}
\savebox{\XtCzero}{
\rput(0,0){\usebox{\Xalone}}
\rput(2,2){\usebox{\Xalone}}
\psline(3,0)(5,2)
\psline(0,0)(2,2)
\psline[linecolor=gray, linewidth=2px](1.5,-0.2)(3.5,1.8)
\psline[linecolor=gray, linewidth=0.5px](0.5,0)(2.5,2)
\psline[linecolor=gray, linewidth=0.5px](1,0)(3,2)
\psline[linecolor=gray, linewidth=0.5px](2,0)(4,2)
\psline[linecolor=gray, linewidth=0.5px](2.5,0.2)(4.5,2.2)
}

\newsavebox{\Sspace}
\savebox{\Sspace}{
\rput(0,0){\usebox{\XaloneHole}}
\rput(2,2){\usebox{\XaloneHole}}
\pscurve(1.2,-0.1)(1.25,-0.105)(1.3,-0.1)(1.7,0.1)(2.3,0.5)
(2.5,1.4)(3.4,1.8)(3.6,1.85)(3.65,1.87)(3.7,1.9)

\psline[linecolor=gray, linewidth=0.5px](0.5,0)(2.5,2)
\psline[linecolor=gray, linewidth=0.5px](1,0)(3,2)
\psline[linecolor=gray, linewidth=0.5px](2,0)(4,2)
\psline[linecolor=gray, linewidth=0.5px](2.5,0.2)(4.5,2.2)

}

\newsavebox{\Xpt}
\savebox{\Xpt}{
\psline(0,0)(3,0)
\rput(1.5,0){\textbullet}
\rput(1.5,-0.3){{\scriptsize $o$}}
\rput(3.4,0){{\scriptsize $X$}}
}

\newsavebox{\XCzero}
\savebox{\XCzero}{
\psline[linecolor=gray, linewidth=2px](1.5,0)(3.5,2)
\psline[linecolor=gray, linewidth=0.5px](0.5,0)(2.5,2)
\psline[linecolor=gray, linewidth=0.5px](1,0)(3,2)
\psline[linecolor=gray, linewidth=0.5px](2,0)(4,2)
\psline[linecolor=gray, linewidth=0.5px](2.5,0)(4.5,2)
\psline[linecolor=gray](3,0)(5,2)
\psline[linecolor=gray](0,0)(2,2)
\psline(0,0)(3,0)
\psline(2,2)(5,2)
\rput(2.6,0.9){{\scriptsize $C$}}

\rput(5,0.2){{\scriptsize $U=\mathbb{A}^1_k\times C$}}
}

\newsavebox{\Xspace}
\savebox{\Xspace}{
\pscurve(1.2,0)(2.3,0.5)(2.8,1.6)(3.7,2)
\psframe[fillstyle=solid, fillcolor=white, linecolor=white](1.48,0.05)(1.72,0.3)
\psframe[fillstyle=solid, fillcolor=white, linecolor=white](2.4,0.9)(2.7,1.15)
\psframe[fillstyle=solid, fillcolor=white, linecolor=white](3.3,1.8)(3.5,2)
\pscurve(1.7,0)(1.6,0.1) (3.4,1.9)(3.3,2)
\psline[linecolor=gray, linewidth=0.5px](0.5,0)(2.5,2)
\psline[linecolor=gray, linewidth=0.5px](1,0)(3,2)
\psline[linecolor=gray, linewidth=0.5px](2,0)(4,2)
\psline[linecolor=gray, linewidth=0.5px](2.5,0)(4.5,2)
\psline[linecolor=gray](3,0)(5,2)
\psline[linecolor=gray](0,0)(2,2)
\psline(0,0)(1.2,0)
\psline(1.7,0)(3,0)
\psline(5,2)(3.7,2)
\psline(3.3,2)(2,2)

\rput(4,0.2){{\scriptsize $V_m/\mu_m$}}
\rput(2.8,0.9){{\scriptsize $\tilde{C}$}}

}

\newsavebox{\XC}
\savebox{\XC}{
\psline[linewidth=2px](1.5,0)(3.5,2)
\psline[linewidth=0.5px](0.5,0)(2.5,2)
\psline[linewidth=0.5px](1,0)(3,2)
\psline[linewidth=0.5px](2,0)(4,2)
\psline[linewidth=0.5px](2.5,0)(4.5,2)
\psline[linewidth=0.5px](3,0)(5,2)
\psline[linewidth=0.5px](0,0)(2,2)
\psline(0,0)(3,0)
\psline(2,2)(5,2)

\rput(2.65,0.9){{\scriptsize $\tilde{C}$}}
\rput(4.6,0.2){{\scriptsize $\tilde{U}=\mathbb{A}^1_k\times \tilde{C}$}}
}

\newsavebox{\XZ}
\savebox{\XZ}{
\psline(1.3,0)(1.6,0.2)(3.6,2.2)(3.3,2) 
\psline(1.7,0)(1.4,-0.2)(3.4,1.8)(3.7,2)
\psline[linewidth=2px](1.6,0.2)(3.6,2.2)
\psline[linewidth=2px](1.4,-0.2)(3.4,1.8)
\psline[linewidth=0.5px](0.5,0)(2.5,2)
\psline[linewidth=0.5px](1,0)(3,2)
\psline[linewidth=0.5px](2,0)(4,2)
\psline[linewidth=0.5px](2.5,0)(4.5,2)
\psline[linewidth=0.5px](3,0)(5,2)
\psline[linewidth=0.5px](0,0)(2,2)
\psline(0,0)(1.3,0)
\psline(1.7,0)(3,0)
\psline(2,2)(3.3,2)
\psline(3.7,2)(5,2)

\rput(3.6,0.1){{\scriptsize $V_m$}}

}

\author{Adrien Dubouloz}

\address{IMB UMR5584, CNRS, Univ. Bourgogne Franche-Comté, F-21000 Dijon,
France.}

\email{adrien.dubouloz@u-bourgogne.fr}

\thanks{This work received support from ANR Project FIBALGA ANR-18-CE40-0003-01
and the \textquotedbl{}Investissements d'Avenir\textquotedbl{} program,
project ISITE-BFC (contract ANR-lS-IDEX-OOOB). }


\title{Exotic $\mathbb{G}_{a}$-quotients of $\mathrm{SL}_{2}\times\mathbb{A}^{1}$ }
\begin{abstract}
Every deformed Koras-Russell threefold of the first kind $Y=\left\{ x^{n}z=y^{m}-t^{r}+xh(x,y,t)\right\} $
in $\mathbb{A}^{4}$ is the algebraic quotient of proper Zariski locally
trivial $\mathbb{G}_{a}$-action on $\mathrm{SL}_{2}\times\mathbb{A}^{1}$. 
\end{abstract}

\maketitle

\section*{Introduction}

Deformed Koras-Russell threefolds (of the first kind) were introduced
in \cite{DMP12} as a familly of smooth affine threefolds generalizing
the famous Koras-Russell threefolds (of the first kind) \cite{KaML97,KoRu97}. 
\begin{defn*}
A \emph{deformed Koras-Russell threefold of the first kind} over an
algebraically closed field $k$ of characteristic zero is a smooth
affine threefold $Y$ isomorphic to a hypersurface $Y(m,n,r,h)$ of
$\mathbb{A}_{k}^{4}=\mathrm{Spec}(k[x,y,z,t]$ defined by an equation
of the form 
\[
x^{n}z=y^{m}-t^{r}+xh(x,y,t),
\]
where $n\geq2$, $m,r\geq1$ are coprime integers, and where $h(x,y,t)\in k[x,y,t]$
is a polynomial such that $h(0,0,0)\in k^{*}$. 
\end{defn*}
All these threefolds share the property to come equipped with a flat
fibration $\mathrm{pr}_{x}:Y\rightarrow\mathbb{A}_{k}^{1}$ restricting
to a trivial $\mathbb{A}^{2}$-bundle $(\mathbb{A}_{k}^{1}\setminus\{0\})\times\mathbb{A}_{k}^{2}=\mathrm{Spec}(k[x^{\pm1}][y,t])$
over $\mathbb{A}_{k}^{1}\setminus\{0\}$ and whose fiber over $\{0\}$
is reduced, isomorphic to the product of the irreducible rational
curve $C=\{y^{m}-t^{r}=0\}\subset\mathbb{A}_{k}^{2}$ with $\mathbb{A}_{k}^{1}=\mathrm{Spec}(k[z])$.
If $m$ or $r$ is equal to $1$, then $C\cong\mathbb{A}_{k}^{1}$
and $\mathrm{pr}_{x}:Y\rightarrow\mathbb{A}_{k}^{1}$ is isomorphic
to the trivial $\mathbb{A}_{k}^{2}$-bundle \cite{Sat83}, and hence
$Y$ is isomorphic to the affine space $\mathbb{A}_{k}^{3}$. Otherwise,
if $m,r\geq2$ then $\mathrm{pr}_{x}^{-1}(0)$ is not isomorphic to
$\mathbb{A}_{k}^{2}$ so that so that $Y$ cannot be isomorphic to
$\mathbb{A}_{k}^{3}$ by \cite{Ka02}. In these cases, it is known
more precisely that the fibration $\mathrm{pr}_{x}:Y\rightarrow\mathbb{A}_{k}^{1}$
is invariant under every algebraic action of the additive group $\mathbb{G}_{a,k}$.
This property turns out to fails for the cylinders $Y\times\mathbb{A}_{k}^{\ell}$,
$\ell\geq1$, with the consequence that the known invariants associated
to $\mathbb{G}_{a,k}$-actions do no longer suffice to distinguish
cylinders $Y\times\mathbb{A}_{k}^{\ell}$, $\ell\geq1$, over threefolds
$Y$ non isomorphic to $\mathbb{A}_{k}^{3}$ from affine spaces $\mathbb{A}_{k}^{\ell+3}$.
The question whether any deformed Koras-Russell threefold $Y\not\cong\mathbb{A}_{k}^{3}$
has a cylinder $Y\times\mathbb{A}_{k}^{\ell}$ isomorphic to an affine
space is totally open. 

The projection $\mathrm{pr}_{x,y,t}:Y\rightarrow\mathbb{A}_{k}^{3}$
is a birational morphism which represents $Y$ as the affine modification
of $\mathbb{A}_{k}^{3}=\mathrm{Spec}(k[x,y,t])$ with center as the
closed subscheme $Z_{Y}$ with defining ideal $J_{Y}=(x^{n},y^{m}-t^{r}+xh(x,y,t))$
and principal divisor $D_{Y}=\{x^{n}=0\}$ in the sense of \cite{KaZa99}.
Equivalently, the coordinate ring of $Y$ is isomorphic to the quotient
of the Rees algebra 
\[
\bigoplus_{s\geq0}J_{Y}^{s}\cdot\upsilon^{s}\subset k[x,y,t][\upsilon]
\]
of the ideal $J_{Y}\subset k[x,y,t]$ by the ideal generated by $1-x^{n}\upsilon$.
Over the field $\mathbb{C}$ of complex numbers, the fact that the
associated closed subscheme $Z_{Y}$ is supported on the topologically
contractible curve $\left\{ y^{m}-t^{r}=0\right\} $ inside the topologically
contractible divisor $\mathrm{Supp}(D_{Y})=\mathbb{A}_{\mathbb{C}}^{2}$
implies by \cite[Theorem 3.1]{KaZa99} that every deformed Koras-Russell
threefolds is a topologically contractible complex threefold, which
is actually even diffeomorphic to $\mathbb{R}^{6}$ \cite{KoRu97}.
For arbitrary algebraically closed fields $k$ of characteristic zero,
it is known by \cite{DuFa18} that if $h\in k[x]\setminus xk[x]$
then a deformed Koras-Russell threefold $Y(m,n,r,h)\not\cong\mathbb{A}_{k}^{3}$
is contractible in the $\mathbb{A}^{1}$-homotopy category $\mathcal{H}(k)$
of Morel and Voevodsky \cite{MoVo01}. The question whether every
$Y=Y(m,n,r,h)$ is contractible in this category for arbitrary $h\in k[x,y,t]$
such that $h(0,0,0)\in k^{*}$ is open, but it was established recently
in \cite{DuOP18} by a combination of the techniques developed in
\cite{HKO16} and \cite{DuFa18} that every such $Y$ becomes contractible
in $\mathcal{H}(k)$ after a single suspension with the simplicial
circle $S^{1}$. 

One of the steps in the proofs of these contractibility results in
the $\mathbb{A}^{1}$-homotopy category $\mathcal{H}(k)$ consists
in determining the $\mathbb{A}^{1}$-homotopy type of the complement
in $Y$ of the curve 
\[
\ell=\{x=y=t=0\}\cong\mathrm{Spec}(k[z]).
\]
Over the field of complex number, the inclusion $\ell\hookrightarrow Y$
defines a smooth proper embedding of the underlying differential manifold
$\mathbb{R}^{2}$ of $\ell$ into the underlying differentiable manifold
$\mathbb{R}^{6}$ of $Y$. Since every two smooth proper embeddings
of $\mathbb{R}^{2}$ into $\mathbb{R}^{6}$ are ambiently isotopic
\cite[Chapter 8]{Hi76}, it follows that $Y\setminus\ell$ is diffeomorphic
to the complement of $\mathbb{R}^{2}$ embedded into $\mathbb{R}^{6}$
as a linear subspace, hence to $(\mathbb{R}^{4}\setminus\{0\})\times\mathbb{R}^{2}$.
The $\mathbb{A}^{1}$-homotopic counterpart that $Y\setminus\ell$
is $\mathbb{A}^{1}$-weakly equivalent to the complement $\mathbb{A}_{k}^{3}\setminus\mathbb{A}_{k}^{1}\cong(\mathbb{A}_{k}^{2}\setminus\{0\})\times\mathbb{A}_{k}^{1}$
of an affine line $\mathbb{A}_{k}^{1}$ embedded into $\mathbb{A}_{k}^{3}$
as a linear subspace was established in \cite{DuFa18,DuOP18} by constructing
for every $Y$ an explicit $\mathbb{A}^{1}$-weak equivalence between
$Y\setminus\ell$ and $\mathbb{A}_{k}^{3}\setminus\mathbb{A}_{k}^{1}$
in the form of a quasi-affine fourfold $W$ which is simultaneously
the total space of a $\mathbb{G}_{a,k}$-torsor over $Y\setminus\ell$
and $\mathbb{A}_{k}^{3}\setminus\mathbb{A}_{k}^{1}$. 

Since it was not important in the next steps of the constructions
in \emph{loc. cit.}, the precise structure of this variety $W$ was
not elucidated. But since then, it has become a kind of folklore fact
that it should be affine, and actually isomorphic to the product of
$\mathrm{SL}_{2}=\left\{ xv-yu=1\right\} \subset\mathbb{A}_{k}^{4}$
with the affine line $\mathbb{A}_{k}^{1}$, independently of the given
deformed Koras-Russell threefold $Y$. The purpose of this article
is to give a complete and detailed proof of this folklore fact, in
the form of the the following theorem: 
\begin{thm}
\label{thm:MainTheorem}Let $k$ be an algebraically closed field
of characteristic zero and let 
\[
Y=Y(m,n,r,h)=\left\{ x^{n}z=y^{m}-t^{r}+xh(x,y,t)\right\} \subset\mathbb{A}_{k}^{4}
\]
be a deformed Koras-Russell threefold. Then there exists a proper
Zariski locally trivial $\mathbb{G}_{a,k}$-action on $\mathrm{SL}_{2}\times\mathbb{A}_{k}^{1}$
whose algebraic quotient $(\mathrm{SL}_{2}\times\mathbb{A}_{k}^{1})/\!/\mathbb{G}_{a,k}=\mathrm{Spec}(\Gamma(\mathrm{SL}_{2}\times\mathbb{A}_{k}^{1},\mathcal{O}_{\mathrm{SL}_{2}\times\mathbb{A}_{k}^{1}})^{\mathbb{G}_{a,k}})$
is isomorphic to $Y$. Furthermore, the quotient morphism $\mathrm{SL}_{2}\times\mathbb{A}_{k}^{1}\rightarrow Y$
restricts to a $\mathbb{G}_{a,k}$-torsor over $Y\setminus\ell$. 
\end{thm}

This implies in particular that the coordinate rings of all Koras-Russell
threefolds of the first kind can be realized as rings of invariants
of $\mathbb{G}_{a,k}$-actions on the single affine fourfold $\mathrm{SL}_{2}\times\mathbb{A}_{k}^{1}$.
In contrast, it is an open question raised by Freudenburg \cite{GSY05}
whether these can be realized as rings of invariants of $\mathbb{G}_{a,k}$-actions
on the affine space $\mathbb{A}_{k}^{4}$. It is also an open question
whether all proper Zariski locally trivial $\mathbb{G}_{a,k}$-actions
on $\mathbb{A}_{k}^{4}$ are conjugate (see e.g. \cite{DuFinJa14, Ka18}
and the references therein for known partial results on this problem).
Since by \cite[Theorem 1]{DMP12}, for every fixed pair $(m,r)$ with
$m,r\geq2$ and for every fixed big enough $n\geq2$, there exists
uncountably many pairwise non isomorphic deformed Koras-Russell threefolds
$Y(m,n,r,h)$, the above theorem implies that this property fails
very badly for $\mathrm{SL}_{2}\times\mathbb{A}^{1}$: 
\begin{cor}
There exists uncountably many pairwise non-conjugate proper Zariski
locally trivial $\mathbb{G}_{a,k}$-actions on $\mathrm{SL}_{2}\times\mathbb{A}^{1}$. 
\end{cor}

The proof of Theorem \ref{thm:MainTheorem} draws as in \cite{DuFa18,DuOP18}
on the study of categorical quotients of certain $\mathbb{G}_{a,k}$-actions
on deformed Koras-Russell threefolds in the category of algebraic
spaces. The algebraic spaces which come into play are a particular
class of ``non-separated surfaces with an $m$-fold curve'' which
already appeared in the context of the study of proper $\mathbb{G}_{a,k}$-actions
on $\mathbb{A}_{k}^{4}$ in \cite{DuFin14,DuFinJa14} and, for some
special cases, in \cite{Du09} and \cite{DMP11} in relation to the
Zariski Cancellation problem for threefolds. In many respects, these
spaces tend to be natural and necessary replacements in higher dimension
of the non-separated curves first considered by Danielewski \cite{Dan89}
in its famous counter-example to the Cancellation problem in dimension
two, and which became ubiquitous in the study $\mathbb{A}^{1}$-fibered
affine surfaces after the work of Fieseler \cite{Fie94}. With the
hope to make the use of these techniques accessible to a larger community,
we collect various complementary descriptions of these spaces which
can be found disseminated in the literature. 

\section{\label{sec:QuotientSpaces-Exotic} preliminaries: Some non proper
$\mathbb{G}_{a}$-actions on exotic affine $3$-spheres}

Let $m\geq1$ be an integer and let $X_{m}\subset\mathbb{A}_{k}^{4}=\mathrm{Spec}(k[x,y,u,v])$
be the smooth closed sub-variety of dimension $3$ defined by the
equation $x^{m}v-yu=1$. The projection 
\[
\mathrm{pr}_{x,y}:X_{m}\rightarrow\mathbb{A}_{*}^{2}=\mathbb{A}_{k}^{2}\setminus\{(0,0)\}
\]
is a Zariski locally trivial $\mathbb{A}^{1}$-bundle which is the
structure morphism of a Zariski locally trivial $\mathbb{G}_{a,k}$-torsor
for the $\mathbb{G}_{a,k}$-action on $X_{m}$ defined by $t\cdot(x,y,u,v)=(x,y,u+tx^{m},v+ty)$.
For every $m,m'\geq1$, the fiber product $W=X_{m}\times_{\mathbb{A}_{*}^{2}}X_{m'}$
is thus simultaneously the total space of a $\mathbb{G}_{a,k}$-torsor
over $X_{m}$ and $X_{m'}$ via the first and second projection respectively.
Since $X_{m}$ and $X_{m'}$ are affine, the vanishing of $H^{1}(X_{m},\mathcal{O}_{X_{m}})$
and $H^{1}(X_{m'},\mathcal{O}_{X_{m'}})$ implies that these two $\mathbb{G}_{a,k}$-torsors
are trivial so that we get isomorphisms 
\begin{equation}
X_{m}\times\mathbb{A}_{k}^{1}\cong W\cong X_{m'}\times\mathbb{A}_{k}^{1}.\label{eq:Xm-SL2-cylinders}
\end{equation}
In particular, for every $m\geq1$, $X_{m}\times\mathbb{A}_{k}^{1}$
is isomorphic to $X_{1}\times\mathbb{A}_{k}^{1}\cong\mathrm{SL}_{2}\times\mathbb{A}_{k}^{1}$. 
\begin{rem}
Over the field of complex numbers $\mathbb{C}$, the underlying $6$-dimensional
real smooth manifold $X_{m}^{\mathrm{diff}}$ of $X_{m}$ is diffeomorphic
to that of $X_{1}$ for every $m\geq1.$ Moreover, $X_{1}^{\mathrm{diff}}$
is homotopically equivalent to the sphere $S^{3}\subset\mathbb{R}^{4}$.
It was established in \cite{DuFin14-S} that for every $m>1$, $X_{m}$
is not isomorphic to $X_{1}$ as an algebraic variety. The threefolds
$X_{m}$, $m>1$, were consequently named exotic affine $3$-spheres.
\end{rem}

Each threefold $X_{m}$ carries another fixed point free $\mathbb{G}_{a,k}$-action
$\nu_{m}:\mathbb{G}_{a,k}\times X_{m}\rightarrow X_{m}$ defined by
the locally nilpotent $k[y,v]$-derivation 
\[
\partial_{m}=y\frac{\partial}{\partial x}+mx^{m-1}v\frac{\partial}{\partial u}
\]
of its coordinate ring $A_{m}=k[x,y,u,v]/(x^{m}v-yu-1)$. The projection
\[
q_{m}=\mathrm{pr}_{y,v}:X_{m}\rightarrow\mathbb{A}_{*}^{2}=\mathrm{Spec}(k[y,v])\setminus\{(0,0)\}
\]
is a smooth $\mathbb{G}_{a,k}$-invariant morphism which restricts
to the trivial $\mathbb{G}_{a,k}$-torsor over the principal affine
open subset $\mathbb{A}_{k,y}^{2}=\mathrm{Spec}(k[y^{\pm1},v))$ of
$\mathbb{A}_{*}^{2}$.

\subsection{Categorical quotients in the category of schemes }

If $m=1$ then $q_{1}:X_{1}\rightarrow\mathbb{A}_{*}^{2}$ is again
a Zariski locally trivial $\mathbb{G}_{a,k}$-torsor. This is no longer
the case when $m>1$ since then the restriction of $q_{m}$ over the
curve $C\cong\mathrm{Spec}(k[v^{\pm1}])$ in $\mathbb{A}_{*}^{2}$
with equation $y=0$ factors as the composition of the trivial $\mathbb{G}_{a,k}$-torsor
\[
\mathrm{pr}_{x}:X_{m}|_{C}\cong\mathrm{Spec}(k[x,u,v]/(x^{m}v-1))\rightarrow\tilde{C}=\mathrm{Spec}(k[x^{\pm1}])
\]
with the cyclic \'etale cover $f:\tilde{C}\rightarrow C$, $x\mapsto v=x^{-m}$
of order $m$. 
\begin{lem}
\label{lem:CatQuot-Schemes}For every $m\geq1$, the morphism $q_{m}:X_{m}\rightarrow\mathbb{A}_{*}^{2}$
is the categorical quotient of $X_{m}$ by the $\mathbb{G}_{a,k}$-action
$\nu_{m}$ in the category of schemes. 
\end{lem}

\begin{proof}
Since the ring of $\mathbb{G}_{a,k}$-invariant functions on $X_{m}$
is equal to sub-algebra $k[y,v]\subset A_{m}$, it follows that the
composition of $q_{m}:X_{m}\rightarrow\mathbb{A}_{*}^{2}$ with the
open inclusion $\mathbb{A}_{*}^{2}\hookrightarrow\mathrm{Spec}(k[y,v])$
is the categorical quotient $X_{m}\rightarrow X_{m}/\!/\mathbb{G}_{a,k}$
of $X_{m}$ in the category of affine schemes. Furthermore, for every
principal affine open subset $V$ of $\mathbb{A}_{k}^{2}$, the composition
\[
q_{m}:q_{m}^{-1}(V_{0})\rightarrow V_{0}=V\cap\mathbb{A}_{*}^{2}\hookrightarrow V
\]
is the categorical quotient in the category of affine schemes of $q_{m}^{-1}(V_{0})$
by the $\mathbb{G}_{a,k}$-action induced by $\nu_{m}$ .

Now let $Z$ be an arbitrary scheme and let $f:X_{m}\rightarrow Z$
be a $\mathbb{G}_{a,k}$-invariant morphism. Since $X_{m}$ is irreducible,
to show that $f:X_{m}\rightarrow Z$ factorizes as $\tilde{f}\circ q_{m}$
for a unique morphism $\tilde{f}:\mathbb{A}_{*}^{2}\rightarrow Z$,
we may assume without loss of generally that $Z$ is irreducible.
Since $q_{m}:X_{m}\rightarrow\mathbb{A}_{*}^{2}$ is a surjective
smooth morphism, hence in particular a faithfully flat morphism, it
follows from faithfully flat descent that $f$ descends to a morphism
$\tilde{f}:\mathbb{A}_{*}^{2}\rightarrow Z$ if and only if it is
constant on the fibers of $q_{m}$. This is clear for $m=1$ as $q_{1}:X_{1}\rightarrow\mathbb{A}_{*}^{2}$
is a $\mathbb{G}_{a,k}$-torsor. We now consider the case $m\geq2$.
Since the restriction of $q_{m}:X_{m}\rightarrow\mathbb{A}_{*}^{2}$
over the principal affine open subset $\mathbb{A}_{k,y}^{2}=\mathbb{A}_{*}^{2}\setminus C$
of $\mathbb{A}_{*}^{2}$ is a $\mathbb{G}_{a,k}$-torsor, $f$ is
constant on the fibers of $q_{m}|_{X_{m}\setminus q_{m}^{-1}(C)}$,
and it remains to check that $f$ is constant on the fibers of $q_{m}|_{q_{m}^{-1}(C)}:q_{m}^{-1}(C)\rightarrow C$.
Since $\mathbb{G}_{a,k}$ acts on $q_{m}^{-1}(C)\cong\tilde{C}\times\mathbb{A}_{k}^{1}$
by translations on the second factor and $f$ is $\mathbb{G}_{a,k}$-invariant,
the image of $q_{m}^{-1}(C)$ by $f$ is either a point and we are
done, or its closure in $Z$ is a curve $D$ dominated by $\tilde{C}$. 

In the second case, let $U\subset Z$ be an affine open subset such
that $D\cap U$ is not empty. Since $X_{m}$ is affine, $f$ is an
affine morphism. It follows that $f^{-1}(U)$ is a $\mathbb{G}_{a,k}$-invariant
affine open subset of $X_{m}$ such that $f^{-1}(U)\cap q_{m}^{-1}(C)$
is not empty. Since $q_{m}$ is smooth, hence open, $q_{m}(f^{-1}(U))$
is an open subset of $\mathbb{A}_{*}^{2}$ such that $q_{m}(f^{-1}(U))\cap C$
is not empty. Let $V_{0}$ be a principal affine open subset of $\mathbb{A}_{*}^{2}$
contained in $q_{m}(f^{-1}(U))$ and intersecting $C$. Then $q_{m}^{-1}(V_{0})$
is contained in $f^{-1}(U)$. Indeed, first note that by the choice
of $V_{0}$, $q_{m}^{-1}(V_{0}\cap C)\cap f^{-1}(U)$ is not empty.
Since $q_{m}^{-1}(V_{0})$ is affine and $X_{m}$ is separated, $q_{m}^{-1}(V_{0})\cap f^{-1}(U)$
is an affine open subset of $q_{m}^{-1}(V_{0})$. It follows that
$q_{m}^{-1}(V_{0})\setminus(q_{m}^{-1}(V_{0})\cap f^{-1}(U))$ is
either empty or a closed subset of pure codimension one in $q_{m}^{-1}(V_{0})$.
On the other hand, since $f$ is constant on the fibers of $q_{m}|_{X_{m}\setminus q_{m}^{-1}(C)}$,
we have $q_{m}^{-1}(V_{0}\setminus C)=q_{m}^{-1}(V_{0})\setminus q_{m}^{-1}(C)=f^{-1}(U)\setminus q_{m}^{-1}(C)$.
So $q_{m}^{-1}(V_{0})\setminus(q_{m}^{-1}(V_{0})\cap f^{-1}(U))$
is contained in $q_{m}^{-1}(V_{0}\cap C)$. Since $q_{m}^{-1}(V_{0}\cap C)$
is irreducible and of pure codimension one in $q_{m}^{-1}(V_{0})$
and $q_{m}^{-1}(V_{0}\cap C)\cap f^{-1}(U)\neq\emptyset$, it follows
that $q_{m}^{-1}(V_{0})\setminus(q_{m}^{-1}(V_{0})\cap f^{-1}(U))$
is empty. Since $U$ is affine and $V_{0}$ is the categorical quotient
of $q_{m}^{-1}(V_{0})\subseteq f^{-1}(U)$ in the category of affine
schemes, it follows that there exists a unique morphism $\tilde{f}:V_{0}\rightarrow U$
such that $f|_{q_{m}^{-1}(V_{0})}=\tilde{f}\circ q_{m}|_{q_{m}^{-1}(V_{0})}$.
This implies that $f$ is generically constant, hence constant, on
the fibers of $q_{m}|_{q_{m}^{-1}(C)}:q_{m}^{-1}(C)\rightarrow C$
as desired. 
\end{proof}

\subsection{\label{subsec:AlgSpace-Quotient}Categorical quotients in the category
of algebraic spaces}

On the other hand, since the $\mathbb{G}_{a,k}$-action $\nu_{m}$
on $X_{m}$ is fixed point free, it admits a categorical quotient
in the larger category of algebraic spaces, in the form of an \'etale
locally trivial $\mathbb{G}_{a,k}$-torsor $\rho_{m}:X_{m}\rightarrow X_{m}/\mathbb{G}_{a,k}$
over a certain algebraic space $X_{m}/\mathbb{G}_{a,k}$ of finite
type and dimension $2$ (see e.g. \cite[10.4]{LMB00}), which is smooth
as $X_{m}$ is smooth. If $m=1$ then $X_{1}/\mathbb{G}_{a,k}=\mathbb{A}_{*}^{2}$.
But if $m\geq2$, it follows from Lemma \ref{lem:CatQuot-Schemes}
that $X_{m}/\mathbb{G}_{a,k}$ cannot be a scheme. Furthermore, $X_{m}/\mathbb{G}_{a,k}$
is not separated for otherwise, being smooth of dimension $2$ and
of finite type over $k$, it would be a quasi-projective $k$-variety
by Chow Lemma. Since $\rho_{m}:X_{m}\rightarrow X_{m}/\mathbb{G}_{a,k}$
is a $\mathbb{G}_{a,k}$-torsor, this implies that for every $m\geq2$
the injective morphism 
\[
\nu_{m}\times\mathrm{pr}_{2}:\mathbb{G}_{a,k}\times X_{m}\cong X_{m}\times_{X_{m}/\mathbb{G}_{a,k}}X_{m}\rightarrow X_{m}\times X_{m}
\]
is not a closed immersion, hence that the action $\nu_{m}$ is not
proper.

Since $\rho_{m}:X_{m}\rightarrow X_{m}/\mathbb{G}_{a,k}$ is a categorical
quotient in the category of algebraic spaces, the surjective morphism
$q_{m}:X_{m}\rightarrow\mathbb{A}_{*}^{2}$ factors as $q_{m}=\tilde{q}_{m}\circ\rho_{m}$
for a unique surjective morphism $\tilde{q}_{m}:X_{m}/\mathbb{G}_{a,k}\rightarrow\mathbb{A}_{*}^{2}$.
Since the restriction of $q_{m}$ over $\mathbb{A}_{k,y}^{2}=\mathbb{A}_{*}^{2}\setminus C$
is already a $\mathbb{G}_{a,k}$-torsor, $\tilde{q}_{m}$ restricts
to an isomorphism over $\mathbb{A}_{*}^{2}\setminus C$. On the other
hand, since $q_{m}^{-1}(C)\cong\tilde{C}\times\mathbb{A}_{k}^{1}$
on which $\mathbb{G}_{a,k}$ acts by translations on the second factor,
it follows that $\tilde{q}_{m}^{-1}(C)\cong(\tilde{C}\times\mathbb{A}_{k}^{1})/\mathbb{G}_{a,k}\cong\tilde{C}$.
So $X_{m}/\mathbb{G}_{a,k}$ is somehow obtained from $\mathbb{A}_{*}^{2}$
by replacing the closed curve $C=\{y=0\}$ by the total space of the
cyclic \'etale cover $f:\tilde{C}\rightarrow C$, $x\mapsto v=x^{-m}$
of order $m$. 

Let us recall from \cite[§ 1.1]{DuFin14} an explicit construction
of an algebraic space $\mathfrak{S}_{m}$ with this property. Let
$U=\mathbb{A}_{k}^{1}\times C=\mathrm{Spec}(k[y,v^{\pm1}])\subset\mathbb{A}_{*}^{2}$,
$U_{*}=U\setminus(\{0\}\times C)=\mathrm{Spec}(k[y^{\pm1},v^{\pm1}]$,
$\tilde{U}=\mathbb{A}_{k}^{1}\times\tilde{C}=\mathrm{Spec}(k[y,x^{\pm1}])$
and let 
\[
\varphi=\mathrm{id}\times f:\tilde{U}=\mathbb{A}_{k}^{1}\times\tilde{C}\rightarrow\mathbb{A}_{k}^{1}\times C=U
\]
be the \'etale morphism deduced from $f:\tilde{C}\rightarrow C$.
Let $\mathrm{diag}:\tilde{U}\hookrightarrow\tilde{U}\times\tilde{U}$
be the diagonal embedding and let $j:(\tilde{U}\times_{U_{*}}\tilde{U})\setminus\mathrm{Diag}\hookrightarrow\tilde{U}\times\tilde{U}$
be the natural immersion. Then the pair of morphisms 
\begin{equation}
(\mathrm{pr}_{1}\circ(\mathrm{diag}\sqcup j),\mathrm{pr}_{2}\circ(\mathrm{diag}\sqcup j)):R=\tilde{U}\sqcup(\tilde{U}\times_{U_{*}}\tilde{U})\setminus\mathrm{Diag}\rightrightarrows\tilde{U}\label{eq:Etale-equiv-rel}
\end{equation}
is an \'etale equivalence relation on $\tilde{U}$. Letting $\tilde{U}/R$
be the algebraic space defined by this \'etale equivalence relation,
it follows that $\varphi:\tilde{U}\rightarrow U$ descends to a morphism
$\beta:\tilde{U}/R\rightarrow U$. By construction, the restriction
$R_{*}$ of $R$ to $\varphi^{-1}(U_{*})=\tilde{U}\setminus(\{0\}\times\tilde{C})$
is equal to the equivalence relation defined by the diagonal embedding
$\varphi^{-1}(U_{*})\hookrightarrow\varphi^{-1}(U_{*})\times_{U_{*}}\varphi^{-1}(U_{*})$,
whose quotient $\varphi^{-1}(U_{*})/R_{*}$ is isomorphic to $U_{*}$.
It follows that $\beta:\tilde{U}/R\rightarrow U$ restricts to an
isomorphism over $U_{*}$. On the other hand, since $R$ restricts
to the trivial equivalence relation on the closed subset $\{0\}\times\tilde{C}\subset\tilde{U}$,
it follows that $\beta^{-1}(\{0\}\times C)\cong\{0\}\times\tilde{C}$. 

Now we let $\mathfrak{S}_{m}$ be the algebraic space obtained by
gluing $\mathbb{A}_{k,y}^{2}=\mathrm{Spec}(k[y^{\pm1},v]$ and $\tilde{U}/R$
along the open subsets $U_{*}=$$\mathrm{Spec}(k[y^{\pm1},v^{\pm1}]$
and $\beta^{-1}(U_{*})$ by the isomorphism $\beta^{-1}(U_{*})\cong U_{*}$
induced by $\beta$. Then there exists a unique morphism $\delta_{m}:\mathfrak{S}_{m}\rightarrow\mathbb{A}_{*}^{2}$
whose restrictions to the corresponding open subsets $\mathbb{A}_{k,y}^{2}$
and $\tilde{U}/R$ of $\mathfrak{S}_{m}$ are equal to the open inclusion
$\mathbb{A}_{k,y}^{2}\hookrightarrow\mathbb{A}_{*}^{2}$ and the composition
of $\beta$ with the open inclusion $U=\mathbb{A}_{k,v}^{2}\hookrightarrow\mathbb{A}_{*}^{2}$
respectively. 
\begin{prop}
\label{prop:CatQuotientSpace-Xm}For every $m\geq1$, $\tilde{q}_{m}:X_{m}/\mathbb{G}_{a,k}\rightarrow\mathbb{A}_{*}^{2}$
and $\delta_{m}:\mathfrak{S}_{m}\rightarrow\mathbb{A}_{*}^{2}$ are
isomorphic algebraic spaces over $\mathbb{A}_{*}^{2}$. 
\end{prop}

\begin{proof}
By construction, $\mathfrak{S}_{1}$ is isomorphic to $\mathbb{A}_{*}^{2}$.
If $m\geq2$, then since $\mathbb{A}_{*}^{2}$ is covered by the principal
affine open subsets $\mathbb{A}_{k,y}^{2}$ and $U=\mathbb{A}_{k,v}^{2}$,
it suffices to show that there exists local isomorphism $\delta_{m}^{-1}(\mathbb{A}_{k,y}^{2})\cong\tilde{q}_{m}^{-1}(\mathbb{A}_{k,y}^{2})$
and $\tilde{q}_{m}^{-1}(U)\cong\delta_{m}^{-1}(U)$ which coincide
over $\mathbb{A}_{k,y}^{2}\cap U=U_{*}$. By construction, we already
have isomorphisms $\delta_{m}^{-1}(\mathbb{A}_{k,y}^{2})\cong\mathbb{A}_{k,y}^{2}\cong\tilde{q}_{m}^{-1}(X_{m}/\mathbb{G}_{a,k})$
as schemes over $\mathbb{A}_{k,y}^{2}$. It remains to construct a
compatible isomorphism $\tilde{q}_{m}^{-1}(U)\cong\delta_{m}^{-1}(U)=\tilde{U}/R$
of algebraic spaces over $U=\mathbb{A}_{k,v}^{2}$. Consider the morphism
\[
\begin{array}{ccc}
\Phi:\tilde{U}\times\mathbb{G}_{a,k}=\mathrm{Spec}(k[x^{\pm1},y])\times\mathbb{G}_{a,k} & \longrightarrow & X_{m}|_{U}\\
((x,y),t) & \mapsto & \nu_{m}(t,(x,y,0,x^{-m}))=(x+ty,y,P(x,y,t),x^{-m})
\end{array}
\]
where 
\[
P(x,y,t)=\sum_{n\geq1}\frac{\partial_{m}^{n}u}{n!}|_{v=x^{-m}}t^{n}=\sum_{n=1}^{m}\frac{m!}{(m-n)!n!}x^{-n}y^{n-1}t^{n}=\sum_{n=1}^{m}\binom{m}{n}(x^{-1}t)^{n}y^{n-1}.
\]
By definition, $\Phi$ is $\mathbb{G}_{a,k}$-equivariant for the
action by translations on the second factor on $\tilde{U}\times\mathbb{G}_{a,k}$
and the action $\nu_{m}$ on $X_{m}|_{U}$. Furthermore, since $\frac{\partial P}{\partial t}(x,0,t)=mx^{-1}$
does not vanish on $\tilde{U}\times\mathbb{G}_{a,k}$, it follows
that the Jacobian matrix 
\[
J(\Phi)=\left(\begin{array}{ccc}
1 & t & y\\
0 & 1 & 0\\
\frac{\partial P}{\partial x} & \frac{\partial P}{\partial y} & \frac{\partial P}{\partial t}\\
-mx^{-m-1} & 0 & 0
\end{array}\right)
\]
of $\Phi$ has rank $3$ at every point of $\tilde{U}\times\mathbb{G}_{a,k}$.
So $\Phi:\tilde{U}\times\mathbb{G}_{a,k}\rightarrow X_{m}|_{U}$ is
an \'etale trivialization of the restriction of the $\mathbb{G}_{a,k}$-action
$\nu_{m}$ on $X_{m}|_{U}$. The coordinate ring $B$ of the fiber
product $(\tilde{U}\times\mathbb{G}_{a,k})\times_{X_{m}|_{U}}(\tilde{U}\times\mathbb{G}_{a,k})$
is isomorphic to the quotient of $k[x_{1}^{\pm1},x_{2}^{\pm1},y,t_{1},t_{2}]$
by the ideal $I$ generated by the elements 
\[
x_{1}^{-m}-x_{2}^{-m},\;x_{1}-x_{2}+y(t_{1}-t_{2}),\;\textrm{and}\;P(x_{1},y,t_{1})-P(x_{2},y,t_{2}).
\]
Writing $x_{1}^{-m}-x_{2}^{-m}=(x_{1}^{-1}-x_{2}^{-1})R(x_{1}^{-1},x_{2}^{-1})$,
$B$ decomposes as the product of the rings 

\[
B_{0}\cong k[x_{1}^{\pm1},x_{2}^{\pm1}][y,t_{1},t_{2}]/(x_{1}^{-1}-x_{2}^{-1},x_{1}-x_{2}+y(t_{1}-t_{2}),P(x_{1},y,t_{1})-P(x_{2},y,t_{2}))
\]
and 
\[
B_{1}=k[x_{1}^{\pm1},x_{2}^{\pm1}][y,t_{1},t_{2}]/(R(x_{1}^{-1},x_{2}^{-1}),x_{1}-x_{2}+y(t_{1}-t_{2}),\;P(x_{1},y,t_{1})-P(x_{2},y,t_{2})).
\]
Since $P(x_{1},y,t_{1})-P(x_{1},y,t_{2})=x_{1}^{-1}(t_{1}-t_{2})(1+yS(x_{1}^{-1},y,t_{1},t_{2}))$
it follows that the homomorphism 
\[
k[x^{\pm1},y][t]\rightarrow B_{0},\,(x,y,t)\mapsto(x_{1},y,t_{1})
\]
is an isomorphism. On the other hand, since $x_{1}-x_{2}$ in invertible
in $k[x_{1}^{\pm1},x_{2}^{\pm1}]/(R(x_{1}^{-1},x_{2}^{-1}))$, $y$
is invertible in $B_{1}$ and we get an isomorphism 
\[
k[x_{1}^{\pm1},x_{2}^{\pm1}]/(R(x_{1}^{-1},x_{2}^{-1}))[y^{\pm1}][t]\rightarrow B_{1},\,(x_{1},x_{2},y,t)\mapsto(x_{1},x_{2},y,t_{1}).
\]
Summing up, $(\tilde{U}\times\mathbb{G}_{a,k})\times_{X_{m}|_{U}}(\tilde{U}\times\mathbb{G}_{a,k})$
is $\mathbb{G}_{a,k}$-equivariantly isomorphic to the disjoint union
of 
\[
\mathrm{Spec}(B_{0})\cong\tilde{U}\times\mathbb{G}_{a,k}\quad\textrm{and}\quad\mathrm{Spec}(B_{1})\cong(\tilde{U}\times_{U_{*}}\tilde{U})\setminus\mathrm{Diag}\times\mathbb{G}_{a,k}
\]
on which $\mathbb{G}_{a,k}$ acts by translations on the second factors,
and we get a cartesian square of \'etale equivalence relations  \[\xymatrix{ R\times \mathbb{G}_{a,k} \cong (\tilde{U}\times\mathbb{G}_{a,k})\times_{X_{m}|_{U}}(\tilde{U}\times\mathbb{G}_{a,k}) \ar@<0.5ex>[rr]^-{\mathrm{pr}_1} \ar@<-0.5ex>[rr]_-{\mathrm{pr}_2} \ar[d]_{\mathrm{pr}_R} & & \tilde{U}\times\mathbb{G}_{a,k}  \ar[d]^{\mathrm{pr}_{\tilde{U}}} \\ R \ar@<0.5ex>[rr]^{\mathrm{pr}_{1}\circ(\mathrm{diag}\sqcup j)} \ar@<-0.5ex>[rr]_{\mathrm{pr}_{2}\circ(\mathrm{diag}\sqcup j)} & &\tilde{U} }\]where
$\mathrm{diag}\sqcup j$ is the morphism defined in (\ref{lem:CatQuot-Schemes})
and where the vertical morphisms are trivial $\mathbb{G}_{a,k}$-torsors.
By \cite[I.5.8]{Knu71}, the right-hand side morphism descends to
a $\mathbb{G}_{a,k}$-torsor 
\[
\pi:X_{m}|_{U}=(\tilde{U}\times\mathbb{G}_{a,k})/(R\times\mathbb{G}_{a,k})\rightarrow\tilde{U}/R=\delta_{m}^{-1}(U).
\]
Since $\rho_{m}|_{U}:X_{m}|_{U}\rightarrow\tilde{q}_{m}^{-1}(U)$
is also by definition a $\mathbb{G}_{a,k}$-torsor, it follows that
there exists a unique isomorphism $\alpha:\tilde{U}/R\stackrel{\cong}{\rightarrow}\tilde{q}_{m}^{-1}(U)$
such that $\rho_{m}|_{U}=\alpha\circ\pi$. This completes the proof. 
\end{proof}

\subsection{Another description of the algebraic space quotients }

An alternative complementary description of the algebraic space $\beta:\tilde{U}/R\rightarrow U$
constructed in subsection \ref{subsec:AlgSpace-Quotient} was given
in \cite{DuFin14} in a more general context. Since this description
is sometimes more convenient to use in practice, let us review it
in detail in our particular situation. We use the notation of subsection
\ref{subsec:AlgSpace-Quotient}. 

The Galois group $\mu_{n}=\mathrm{Spec}(k[\varepsilon]/(\varepsilon^{m}-1))$
of $m$-th roots of unity in $k^{*}$ acts on the finite \'etale
cover 
\[
f:\tilde{C}=\mathrm{Spec}(k[x^{\pm1}])\rightarrow C=\mathrm{Spec}(k[v^{\pm1}]),\,x\mapsto v=x^{-m}
\]
by $x\mapsto\varepsilon x$. We let $V_{m}$ be the scheme obtained
by gluing $m$ copies $\tilde{U}_{i}$, $i\in\mathbb{Z}/m\mathbb{Z}$,
of $\tilde{U}=\mathrm{Spec}(k[y,x^{\pm1}])=\mathbb{A}_{k}^{1}\times\tilde{C}$
by the identity outside the curves $\{0\}\times\tilde{C}\subset\tilde{U}_{i}$.
The group $\mu_{m}$ now acts freely on $V_{m}$ by 
\[
\tilde{U}_{i}\ni(y,x)\mapsto(y,\varepsilon x)\in\tilde{U}_{i+1},\quad i\in\mathbb{Z}/m\mathbb{Z},
\]
and the local isomorphisms $\mathrm{id}:\tilde{U}_{i}\rightarrow\tilde{U},$
glue to a global morphism $\pi_{m}:V_{m}\rightarrow\tilde{U}$ which
is equivariant for the $\mu_{m}$-action $(y,x)\mapsto(y,\varepsilon x)$
on $\tilde{U}$. Since the so-defined $\mu_{m}$-action on $V_{m}$
has trivial isotropies, a quotient $\xi:V_{m}\rightarrow V_{m}/\mu_{m}$
exists in the category of algebraic spaces in the form of an \'etale
$\mu_{m}$-torsor over a certain algebraic space $V_{m}/\mu_{m}$.
The $\mu_{m}$-invariant morphism $\gamma_{m}=(\mathrm{id}\times f)\circ\pi_{m}:V_{m}\rightarrow U$
descends to a morphism $\overline{\gamma}_{m}:V_{m}/\mu_{m}\rightarrow U$
which restricts to an isomorphism over $U_{*}=U\setminus(\{0\}\times C)$.
In contrast, $\overline{\gamma_{m}}^{-1}(\{0\}\times C)$ is isomorphic
as a scheme over $C$ to the quotient of $\mu_{m}\times\tilde{C}$
by the diagonal action of $\mu_{m}$, hence to $\tilde{C}$. 

\begin{figure}[ht]
\psset{unit=0.8}
\begin{pspicture}(12,-1)(10,8)
\rput(1,4){\usebox{\XZ}}
\psline{->}(4,3.8)(4,2.8)
\rput(4.5,3.3){{\scriptsize $\pi_m$}}
\rput(0,0){\usebox{\XC}}
\rput(8,0){\usebox{\Xspace}}
\rput(16,0){\usebox{\XCzero}}
\psline{->}(6.45,4.8)(10,2.6)
\rput(8.5,3.8){{\scriptsize $\xi$}}
\psline{->}(6.5,5)(18,2.6)
\rput(13,3.9){{\scriptsize $\gamma_m$}}
\psline{->}(14,1)(16.5,1)
\rput(15.2,1.3){{\scriptsize $\bar{\gamma}_m$}}
\end{pspicture}
\caption{Construction of $\tilde{U}/R $ as a quotient of $V_m$ by a free $\mu_m$-action}
\end{figure}
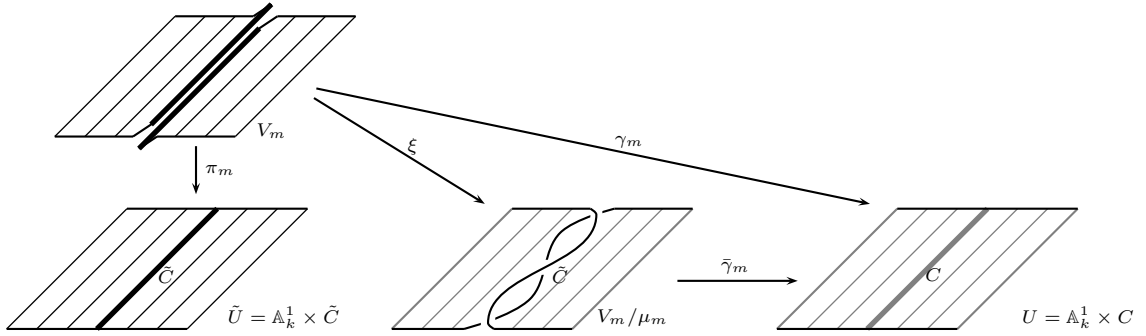
\begin{lem}
\label{lem:QuotientSpace-GaloisAction}The algebraic space $\overline{\gamma}_{m}:V_{m}/\mu_{m}\rightarrow U$
is $U$-isomorphic to $\beta:\tilde{U}/R\rightarrow U$. 
\end{lem}

\begin{proof}
Indeed, letting $\alpha:\mu_{m}\times V_{m}\rightarrow V_{m}$ be
the $\mu_{m}$-action on $V_{m}$, the algebraic space $V_{m}/\mu_{m}$
is by definition the quotient of $V_{m}$ by the \'etale equivalence
relation $(\alpha,\mathrm{pr}_{2}):\mu_{m}\times V_{m}\rightrightarrows V_{m}$.
It is straightforward to check with the definition of $\alpha$ and
$R$ that the composition $\sigma:V_{m}\rightarrow\tilde{U}/R$ of
$\pi_{m}:V_{m}\rightarrow\tilde{U}$ with the quotient morphism $\tilde{U}\rightarrow\tilde{U}/R$
is a quasi-finite $\mu_{m}$-invariant morphism, which descends to
a bijective quasi-finite $U$-morphism $\overline{\sigma}:V_{m}/\mu_{m}\rightarrow\tilde{U}/R$.
To prove that $\overline{\sigma}$ is an isomorphism, it now suffices
to construct a section of it. Let $i_{0}:\tilde{U}\rightarrow V_{m}$
be the section of $\pi_{m}$ defined by the inclusion of $\tilde{U}$
as the open subset $\tilde{U}_{0}\subset V_{m}$. Since the restriction
of $\mathrm{id}\times f:\tilde{U}\rightarrow U$ over $U_{*}$ is
a finite \'etale $\mu_{m}$-cover, there exists an isomorphism 
\[
\tilde{U}\times_{U_{*}}\tilde{U}\stackrel{\cong}{\rightarrow}(\tilde{U}\setminus(\{0\}\times\tilde{C}))\times\mu_{m}
\]
which maps the diagonal $\mathrm{Diag}$ onto $\tilde{U}\setminus(\{0\}\times\tilde{C})\times\{1\}$.
This yields an isomorphism 
\[
(\tilde{U}\times_{U_{*}}\tilde{U})\setminus\mathrm{Diag}\stackrel{\cong}{\rightarrow}(\tilde{U}\setminus(\{0\}\times\tilde{C}))\times(\mu_{m}\setminus\{1\})
\]
hence an open embedding
\[
\zeta_{0}:R=\tilde{U}\sqcup(\tilde{U}\times_{U_{*}}\tilde{U})\setminus\mathrm{Diag}\hookrightarrow V_{m}\times\{1\}\sqcup V_{m}\times(\mu_{m}\setminus\{1\})
\]
whose image is the union of the open subsets $\tilde{U}_{0}\times\{1\}$
of $V_{m}\times\{1\}$ and $(\tilde{U}_{0}\setminus(\{0\}\times\tilde{C}))\times(\mu_{m}\setminus\{1\})$
of $V\times(\mu_{m}\setminus\{1\})$. By construction of $\zeta_{0}$,
the diagram \[\xymatrix{R=\tilde{U}\sqcup (\tilde{U} \times_{U_*} \tilde{U})\setminus\mathrm{Diag} \ar[rr]^-{\mathrm{diag}\sqcup j} \ar[d]_{\zeta_0} & & \tilde{U}\times\tilde{U} \ar[d]^{i_0\times i_0} \\ V\times\mu_m \ar[rr]^{\alpha\times \mathrm{pr_2}} & & V\times V}\] 
is cartesian. It follows that the section $i_{0}:\tilde{U}\rightarrow V_{m}$
of $\pi_{m}:V_{m}\rightarrow\tilde{U}$ descends to a $U$-morphism
$\tilde{U}/R\rightarrow V_{m}/\mu_{m}$ which is a section of $\overline{\sigma}$. 
\end{proof}
A practical consequence of Lemma \ref{lem:QuotientSpace-GaloisAction}
is the following: 
\begin{criterion}
\label{cri:Criterion-descent}A morphism $\tau:Y\rightarrow U$ from
a scheme $Y$ factors through a morphism $\tilde{\tau}:Y\rightarrow\tilde{U}/R$
if and only if the $\mu_{m}$-equivariant morphism $\mathrm{pr}_{2}:Y\times_{U}\tilde{U}\rightarrow\tilde{U}$
lifts to a $\mu_{m}$-equivariant morphism $\tilde{\mathrm{pr}}_{2}:Y\times_{U}\tilde{U}\rightarrow V_{m}$
such that $\mathrm{pr}_{2}=\pi_{m}\circ\tilde{\mathrm{pr}_{2}}.$ 
\end{criterion}

\subsection{An application}

To finish this section, let us give a first concrete application of
Criterion \ref{cri:Criterion-descent}. Given integers $m,n,r\geq1$
such that $\gcd(m,r)=1$, we let $X(m,n,r)$ be the smooth threefold
in $\mathbb{A}_{k}^{4}=\mathrm{Spec}(k[x,y,u,v])$ defined by the
equation 
\[
x^{m}v^{r}-y^{n}u=1.
\]
We thus have $X(m,1,1)=X_{m}$. The locally nilpotent $k[y,v]$-derivation
\[
\partial=y^{n}\frac{\partial}{\partial x}+mx^{m-1}v^{r}\frac{\partial}{\partial u}
\]
of the coordinate ring of $X(m,n,r)$ defines a fixed point free $\mathbb{G}_{a,k}$-action
on $X(m,n,r)$. As for $X_{m}$, the projection 
\[
q_{(m,n,r)}=\mathrm{pr}_{y,v}:X(m,n,r)\rightarrow\mathbb{A}_{*}^{2}=\mathrm{Spec}(k[y,v])\setminus\{(0,0)\}
\]
is a smooth $\mathbb{G}_{a,k}$-invariant morphism which restricts
to the trivial $\mathbb{G}_{a,k}$-torsor over the principal affine
open subset $\mathbb{A}_{k,y}^{2}=\mathrm{Spec}(k[y^{\pm1},v))$ of
$\mathbb{A}_{*}^{2}$. 
\begin{prop}
\label{prop:Xmnr-ALgQuotient}The morphism $q_{(m,n,r)}:X(m,n,r)\rightarrow\mathbb{A}_{*}^{2}$
factors through an \'etale locally trivial $\mathbb{G}_{a,k}$-torsor
$\rho_{(m,n,r)}:X(m,n,r)\rightarrow\mathfrak{S}_{m}$. 
\end{prop}

\begin{proof}
By construction of $\mathfrak{S}_{m}$ as the gluing of $\tilde{U}/R$
and $\mathbb{A}_{k,y}^{2}$ described in subsection \ref{subsec:AlgSpace-Quotient},
it is enough to check as in the case of $X(m,1,1)$ that $q_{(m,n,r)}|_{U}:X(m,n,r)|_{U}\rightarrow U$
factors through a $\mathbb{G}_{a,k}$-torsor over $\tilde{U}/R\cong V_{m}/\mu_{m}$.
Let $\tilde{U}=\mathrm{Spec}(k[y,\lambda^{\pm1}])\rightarrow U$ be
the \'etale $\mu_{m}$-cover defined by $(y,\lambda)\mapsto(y,\lambda^{-m})$.
The fiber product $\tilde{Y}=Y\times_{U}\tilde{U}$ is isomorphic
to the closed subscheme in $\tilde{U}\times\mathrm{Spec}(k[x,u])$
defined by the equation 
\[
y^{n}u=(\lambda^{-r}x)^{m}-1
\]
on which $\mu_{m}$ acts by $(y,\lambda,x,u)\mapsto(y,\varepsilon\lambda,x,u$),
where $\varepsilon\in k^{*}$ is a primitive $m$-th root of unity.
The induced $\mathbb{G}_{a,k}$-action on $X(m,n,r)|_{U}$ lifts on
$\tilde{Y}$ to the $\mathbb{G}_{a,k}$-action commuting with the
action of $\mu_{m}$ defined by the locally nilpotent $k[y,\lambda^{\pm1}]$-derivation
\[
\tilde{\partial}=y^{n}\frac{\partial}{\partial x}+mx^{m-1}\lambda^{r}\frac{\partial}{\partial u}.
\]
Since $\gcd(m,r)=1$, the inverse image by $\mathrm{pr}_{2}:\tilde{Y}\rightarrow\tilde{U}$
of the curve $\tilde{C}=\{y=0\}\cong\mathrm{Spec}(k[\lambda^{\pm1}])$
is the disjoint union of $m$ irreducible surfaces $\tilde{S}_{i}=\tilde{C}_{i}\times\mathrm{Spec}(k[u])$,
where 
\[
\tilde{C}_{i}=\mathrm{Spec}(k[\lambda^{\pm1},x]/(\lambda^{-r}x-\varepsilon^{ri}))\cong\mathrm{Spec}(k[\lambda^{\pm1}]),\;i\in\mathbb{Z}/m\mathbb{Z}.
\]
Furthermore, the group $\mu_{m}$ acts transitively on $\mathrm{pr}_{2}^{-1}(\tilde{C})$
by $\tilde{S}_{i}\ni(\lambda,u)\mapsto(\varepsilon\lambda,u)\in\tilde{S}_{i+1}$.
For every $i\in\mathbb{Z}/m\mathbb{Z}$, the $\tilde{U}$-morphism
\[
\tilde{U}_{i}\times\mathbb{A}_{k}^{1}=\mathrm{Spec}(k[y,\lambda^{\pm1}][v_{i}])\rightarrow\tilde{Y},\;(y,\lambda,v_{i})\mapsto(y,\lambda,y^{n}v_{i}+(\varepsilon^{i}\lambda)^{r},\lambda^{-mr}v_{i}\prod_{j\neq i}(y^{n}v_{i}+\lambda^{r}(\varepsilon^{ri}-\varepsilon^{rj}))
\]
induces a $\tilde{U}$-isomorphism between $\tilde{U}_{i}\times\mathbb{A}_{k}^{1}\cong\tilde{U}\times\times\mathbb{A}_{k}^{1}$
and the $\mathbb{G}_{a,k}$-invariant open subset $\tilde{Y}\setminus\bigcup_{j\neq i}\tilde{S}_{j}$
of $\tilde{Y}$. Furthermore, using the expression $v_{i}=y^{-n}(x-(\varepsilon^{i}\lambda)^{r})$
as a rational function on $\tilde{Y}$, we see that $\tilde{\partial}(v_{i})=1$
so that $\tilde{Y}\setminus\bigcup_{j\neq i}\tilde{S}_{j}$ is $\mathbb{G}_{a,k}$-equivarianlty
isomorphic to $\tilde{U}_{i}\times\mathbb{A}_{k}^{1}$ on which $\mathbb{G}_{a,k}$
acts by translations on the second factor. The restriction of $\mathrm{pr}_{2}$
to $\tilde{Y}\setminus\bigcup_{j\neq i}\tilde{S}_{j}$ is thus the
trivial $\mathbb{G}_{a,k}$-torsor over $\tilde{U}_{i}$. It follows
that $\mathrm{pr}_{2}:\tilde{Y}\rightarrow\tilde{U}$ factors through
a $\mathbb{G}_{a,k}$-torsor $\tilde{\mathrm{pr}}_{2}:\tilde{Y}\rightarrow V_{m}$
with gluing isomorphisms defined by 
\[
(y,\lambda^{\pm1},v_{i})\mapsto(y,\lambda^{\pm1},v_{i}+y^{-n}\lambda^{r}(\varepsilon^{ri}-\varepsilon^{rj})).
\]
By construction $\tilde{\mathrm{pr}}_{2}:\tilde{Y}\rightarrow V_{m}$
is equivariant for the $\mu_{m}$-actions on $\tilde{Y}$ and $V_{m}$
respectively. So $\tilde{\mathrm{pr}}_{2}:\tilde{Y}\rightarrow V_{m}$
descends to an \'etale locally trivial $\mathbb{G}_{a,k}$-torsor
$\rho_{(m,n,r),U}:\tilde{Y}/\mu_{m}=X(m,n,r)|_{U}\rightarrow V_{m}/\mu_{m}=\tilde{U}/R$
such that $q_{(m,n,r)}|_{U}=\beta\circ\rho_{(m,n,r),U}$ as desired.
\end{proof}
It seems that the discrete family of threefolds $X(m,n,r)$ has not
been studied yet in the literature. In particular, to the author's
knowledge, the dependence of their isomorphism types in terms of the
parameters $m$, $n$ and $r$ is unknown. The following result implies
that some of these could provide new types of exotic affine $3$-spheres: 
\begin{cor}
For every triple $(m,n,r)$ of positive integers such that $\gcd(m,r)=1$,
the affine fourfold $X(m,n,r)\times\mathbb{A}_{k}^{1}$ is isomorphic
to $\mathrm{SL}_{2}\times\mathbb{A}_{k}^{1}$ . 
\end{cor}

\begin{proof}
Since $\rho_{(m,n,r)}:X(m,n,r)\rightarrow\mathfrak{S}_{m}$ is the
total space of an \'etale locally trivial $\mathbb{G}_{a,k}$-torsor,
it follows that the fiber product $X(m,n,r)\times_{\mathfrak{S}_{m}}X(m,1,1)$
is simultaneously the total space of an \'etale $\mathbb{G}_{a,k}$-torsor
over $X(m,n,r)$ and $X_{m}=X(m,1,1)$ via the first and second projection
respectively. Since $X_{m}$ and $X(m,n,r)$ are affine, these torsors
are the trivial ones, which yields isomorphisms 
\[
X(m,n,r)\times\mathbb{A}_{k}^{1}\cong X(m,n,r)\times_{\mathfrak{S}_{m}}X(m,1,1)\cong X_{m}\times\mathbb{A}_{k}^{1}.
\]
The result follows since on the other hand $X_{m}\times\mathbb{A}_{k}^{1}\cong X_{1}\times\mathbb{A}_{k}^{1}=\mathrm{SL}_{2}\times\mathbb{A}_{k}^{1}$
by (\ref{eq:Xm-SL2-cylinders}). 
\end{proof}

\section{Fixed point free $\mathbb{G}_{a,k}$- actions on punctured deformed
Koras-Russell threefolds}

Every deformed Koras-Russell threefold 
\[
Y=Y(m,n,r,h)=\{x^{n}z=y^{m}-t^{r}+xh(x,y,t)\}\subset\mathbb{A}_{k}^{4}
\]
admits a $\mathbb{G}_{a,k}$-action defined by the locally nilpotent
derivation $k[y,t]$-derivation
\[
x^{n}\frac{\partial}{\partial y}+(my^{m-1}+x\frac{\partial h}{\partial y}(x,y,t))\frac{\partial}{\partial z}
\]
of its coordinate ring. The fixed point locus of this action is equal
to the affine line 
\[
\ell=\{x=y=t=0\}\cong\mathrm{Spec}(k[z]),
\]
so that the action restricts to a fixed point free $\mathbb{G}_{a,k}$-action
on the quasi-affine threefold $Y_{*}=Y\setminus\ell$. 

The proof of Theorem \ref{thm:MainTheorem} we give in the next subsection
then essentially follows from the basic observation that the categorical
quotient $Y_{*}/\mathbb{G}_{a,k}$ taken in the category of algebraic
spaces is isomorphic to the algebraic space $\mathfrak{S}_{m}$ described
in subsection \ref{subsec:AlgSpace-Quotient}. 

\subsection{\label{subsec:Proof-of-Theorem}Proof of Theorem \ref{thm:MainTheorem}}

Since the $\mathbb{G}_{a,k}$-action on $Y_{*}$ defined above is
fixed point free, the categorical quotient $Y_{*}\rightarrow Y_{*}/\mathbb{G}_{a,k}$
exists in the form of an \'etale locally trivial $\mathbb{G}_{a,k}$-torsor
over a certain algebraic space $X_{*}/\mathbb{G}_{a,k}$. The $\mathbb{G}_{a,k}$-invariant
projection $\mathrm{pr}_{x,t}:Y\rightarrow\mathbb{A}_{k}^{2}$ induces
a surjective morphism $\pi:Y_{*}\rightarrow\mathbb{A}_{*}^{2}=\mathbb{A}_{k}^{2}\setminus\{(0,0)\}$,
which restricts further over the principal affine open subset $\mathbb{A}_{k,x}^{2}=\mathrm{Spec}(k[x^{\pm1},t])$
to the trivial $\mathbb{G}_{a,k}$-torsor. On the other hand, the
restriction of $\pi$ over the curve $B\cong\mathrm{Spec}(k[t^{\pm1}])$
in $\mathbb{A}_{*}^{2}$ with equation $x=0$ factors as the composition
of the trivial $\mathbb{G}_{a,k}$-torsor 
\[
\pi:Y|_{B}\cong\mathrm{Spec}(k[y,t^{\pm1},z]/(y^{m}-t^{r}))\rightarrow\tilde{B}=\mathrm{Spec}(k[y,t^{\pm1}]/(y^{m}-t^{r}))
\]
with the projection $f:\tilde{B}\rightarrow B$, $(y,t)\mapsto t$.
Since $\gcd(m,r)=1$, the curves $B$ and $\tilde{B}$ are both isomorphic
to the punctured affine line $\mathbb{A}^{1}\setminus\{0\}$, and
$f$ is a finite cyclic \'etale cover of order $m$. This strongly
suggests that $Y_{*}/\mathbb{G}_{a,k}$ should be isomorphic to $\mathfrak{S}_{m}$,
and this is indeed the case: 
\begin{prop}
\label{prop:CatQuotientSpace-DKR} The categorical quotient $Y_{*}/\mathbb{G}_{a,k}$
in the category of algebraic spaces is isomorphic to $\mathfrak{S}_{m}$.
\end{prop}

\begin{proof}
The result can be extracted from the proof of Lemma 4.6 in \cite{DuOP18}
(see also \cite[Lemma 3.2]{DuFa18}). Let us nevertheless sketch the
main steps which proceed along the same lines as the method employed
in the proof of Proposition \ref{prop:Xmnr-ALgQuotient}. Since the
restriction of $\pi$ over $\mathbb{A}_{k,x}^{2}$ is the trivial
$\mathbb{G}_{a,k}$-torsor, it is again enough to check that the restriction
of $\pi$ over $U=\mathbb{A}_{k,t}^{2}=\mathrm{Spec}(k[x,t^{\pm1}])$
factors through a $\mathbb{G}_{a,k}$-torsor over $\tilde{U}/R\cong V_{m}/\mu_{m}$,
where $\tilde{U}=\mathrm{Spec}(k[x,\lambda^{\pm1}])\rightarrow U$
is \'etale $\mu_{m}$-cover defined by $(x,\lambda)\mapsto(x,\lambda^{m})$.
The fiber product $\tilde{Y}=Y\times_{U}\tilde{U}$ is isomorphic
to the closed subscheme in $\tilde{U}\times\mathrm{Spec}(k[y,z])$
defined by the equation
\[
x^{n}z=y^{m}-\lambda^{mr}+xh(x,y,\lambda^{m}),
\]
on which $\mu_{m}$ acts by $(x,y,\lambda,z)\mapsto(x,y,\varepsilon\lambda,z)$,
where $\varepsilon\in k^{*}$ is a primitive $m$-th root of unity.
The induced $\mathbb{G}_{a,k}$-action on $Y|_{U}$ lifts on $\tilde{Y}$
to the $\mathbb{G}_{a,k}$-action commuting with the action of $\mu_{m}$
defined by the locally nilpotent $k[x,\lambda^{\pm1}]$-derivation
\[
\tilde{\partial}=x^{n}\frac{\partial}{\partial y}+(my^{m-1}+x\frac{\partial h}{\partial y}(x,y,\lambda^{m}))\frac{\partial}{\partial z}
\]
Since $\gcd(m,r)=1$, the inverse image by $\mathrm{pr}_{2}:\tilde{Y}\rightarrow\tilde{U}$
of the curve $\tilde{B}=\{x=0\}\cong\mathrm{Spec}(k[\lambda^{\pm1}])$
is the disjoint union of the $\mathbb{G}_{a,k}$-invariant irreducible
surfaces $\tilde{S}_{i}=\tilde{B}_{i}\times\mathrm{Spec}(k[z])$,
where 
\[
\tilde{B}_{i}=\mathrm{Spec}(k[\lambda^{\pm1},y]/(y-(\varepsilon^{i}\lambda)^{r})\cong\mathrm{Spec}(k[\lambda^{\pm1}]),\;i\in\mathbb{Z}/m\mathbb{Z}.
\]
The restriction of the $\mathbb{G}_{a,k}$-action on each $\tilde{S}_{i}$
is given by the locally nilpotent $k[\lambda^{\pm1}]$-derivation
$m(\varepsilon^{i}\lambda)^{r(m-1)}\frac{\partial}{\partial z}$,
so that the projection $\mathrm{pr}_{\tilde{B}_{i}}:\tilde{S}_{i}\rightarrow\tilde{B}_{i}$
is a trivial $\mathbb{G}_{a,k}$-torsor. Furthermore, the group $\mu_{m}$
acts transitively on $\mathrm{pr}_{2}^{-1}(\tilde{B})$ by $\tilde{S}_{i}\ni(\lambda,z)\mapsto(\varepsilon\lambda,z)\in\tilde{S}_{i+1}$.
For every $i\in\mathbb{Z}/m\mathbb{Z}$, $\mathrm{pr}_{2}$ restricts
on the open subset $\tilde{Y}\setminus\bigcup_{j\neq i}\tilde{S}_{j}$
of $\tilde{Y}$ to a surjective $\mathbb{G}_{a,k}$-invariant smooth
morphism $\mathrm{pr}_{2,i}:\tilde{Y}\setminus\bigcup_{j\neq i}\tilde{S}_{j}\rightarrow\tilde{U}_{i}$
whose fibers each consist of a unique $\mathbb{G}_{a,k}$-orbit. It
follows that $\mathrm{pr}_{2,i}$ is a $\mathbb{G}_{a,k}$-torsor,
hence is isomorphic to the trivial one as $\tilde{U}_{i}\cong\tilde{U}$
is affine. We conclude that $\mathrm{pr}_{2}:\tilde{Y}\rightarrow\tilde{U}$
factors in a unique way through a $\mathbb{G}_{a,k}$-torsor $\tilde{\mathrm{pr}}_{2}:\tilde{Y}\rightarrow V_{m}$
equivariant for the $\mu_{m}$-actions on $\tilde{Y}$ and $V_{m}$
respectively. So $\tilde{\mathrm{pr}}_{2}:\tilde{Y}\rightarrow V_{m}$
descends to an \'etale locally trivial $\mathbb{G}_{a,k}$-torsor
$\tilde{\pi}:\tilde{Y}/\mu_{m}=Y|_{U}\rightarrow V_{m}/\mu_{m}=\tilde{U}/R$
which factors the projection $\pi:Y|_{U}\rightarrow U$. 
\end{proof}
We can now finish the proof of Theorem \ref{thm:MainTheorem} as follows.
By Proposition \ref{prop:CatQuotientSpace-Xm} and Proposition \ref{prop:CatQuotientSpace-DKR},
$X_{m}$ and $Y_{*}$ are \'etale locally trivial $\mathbb{G}_{a,k}$-torsors
over the same algebraic space $\mathfrak{S}_{m}$. This implies that
the fiber product $W_{m}=Y_{*}\times_{\mathfrak{S}_{m}}X_{m}$ is
simultanesouly the total space of \'etale $\mathbb{G}_{a,k}$-torsors
over $Y_{*}$ and $X_{m}$ via the first and second projection respectively.
Since $Y_{*}$ is separated, the $\mathbb{G}_{a,k}$-action on $W_{m}$
corresponding the $\mathbb{G}_{a,k}$-torsor $\mathrm{pr}_{1}:W_{m}\rightarrow Y_{*}$
is proper. Furthermore, since $Y_{*}$ is a scheme, $\mathrm{pr}_{1}:W_{m}\rightarrow Y_{*}$
is in fact locally trivial in the Zariski topology \cite{Gro58}.
On the other hand, since $X_{m}$ is affine $\mathrm{pr}_{2}:W_{m}\rightarrow X_{m}$
is the trivial $\mathbb{G}_{a,k}$-torsor. Thus $W_{m}\cong X_{m}\times\mathbb{A}_{k}^{1}$
and hence $W_{m}\cong\mathrm{SL}_{2}\times\mathbb{A}_{k}^{1}$ by
(\ref{eq:Xm-SL2-cylinders}). The $\mathbb{G}_{a,k}$-action on $W_{m}$
defining the $\mathbb{G}_{a,k}$-torsor $\mathrm{pr}_{1}:W_{m}\rightarrow Y_{*}$
thus corresponds via these isomorphisms to a proper and Zariski locally
trivial $\mathbb{G}_{a,k}$-action on $\mathrm{SL}_{2}\times\mathbb{A}_{k}^{1}$,
whose categorical quotient $(\mathrm{SL}_{2}\times\mathbb{A}_{k}^{1})/\mathbb{G}_{a,k}$
in the category of algebraic spaces is isomorphic to the quasi-affine
variety $Y_{*}=Y\setminus\ell$. 
\begin{lem}
The categorical quotient $(\mathrm{SL}_{2}\times\mathbb{A}_{k}^{1})/\!/\mathbb{G}_{a,k}=\mathrm{Spec}(\Gamma(\mathrm{SL}_{2}\times\mathbb{A}_{k}^{1},\mathcal{O}_{\mathrm{SL}_{2}\times\mathbb{A}_{k}^{1}})^{\mathbb{G}_{a,k}})$
in the category of affine schemes is isomorphic to $Y$. 
\end{lem}

\begin{proof}
The universal properties of the categorical quotient $\mathrm{pr}_{1}:W_{m}\rightarrow W_{m}/\mathbb{G}_{a,k}=Y_{*}$
and the affinization morphism $W_{m}/\mathbb{G}_{m,k}\rightarrow\mathrm{Spec}(\Gamma(W_{m}/\mathbb{G}_{m,k},\mathcal{O}_{W_{m}/\mathbb{G}_{m,k}}))$
imply that 
\[
W_{m}/\!/\mathbb{G}_{a,k}=\mathrm{Spec}(\Gamma(W_{m}/\mathbb{G}_{m,k},\mathcal{O}_{W_{m}/\mathbb{G}_{m,k}}))=\mathrm{Spec}(\Gamma(Y_{*},\mathcal{O}_{Y_{*}})).
\]
Since $Y$ is a smooth, hence normal, affine variety and $\ell$ has
pure codimension $2$ in $Y$, we have $\Gamma(Y_{*},\mathcal{O}_{Y_{*}})=\Gamma(Y,\mathcal{O}_{Y})$.
Thus 
\[
(\mathrm{SL}_{2}\times\mathbb{A}_{k}^{1})/\!/\mathbb{G}_{a,k}\cong W_{m}/\!/\mathbb{G}_{a,k}\cong\mathrm{Spec}(\Gamma(Y,\mathcal{O}_{Y}))=Y.
\]
\end{proof}

\subsection{Complements}

The proof of Theorem \ref{thm:MainTheorem} provides a systematic
method to construct for each given deformed Koras-Russell threefold
$Y=Y(m,n,r,h)$ an explicit locally nilpotent derivation $\delta$
of the coordinate ring of $\mathrm{SL}_{2}\times\mathbb{A}_{k}^{1}$
with kernel isomorphic to $\Gamma(Y,\mathcal{O}_{Y})$, usually at
the cost of a series of tedious calculations. 

To explain the scheme of this method, let $\mathbb{A}_{x,t}^{2}=\mathrm{Spec}(k[x,t])$,
$\mathbb{A}_{x,v}^{2}=\mathrm{Spec}(k[x,v])$, and choose the coordinates
so that $Y=Y(m,n,r,h)$ and $X_{m}$ are given by the equations 
\[
x^{n}z=y^{m}-t^{r}+xh(x,y,t)\quad\textrm{and}\quad v^{m}t-xu=1
\]
in $\mathbb{A}_{x,t}^{2}\times\mathrm{Spec}(k[y,z])$ and $\mathbb{A}_{x,t}^{2}\times\mathrm{Spec}(k[u,v])=\mathbb{A}_{x,v}^{2}\times\mathrm{Spec}(k[u,t])$
respectively. Let 
\[
\partial=x^{n}\frac{\partial}{\partial y}+(my^{m-1}+x\frac{\partial h}{\partial y}(x,y,t))\frac{\partial}{\partial z}\quad\textrm{and}\quad\partial_{m}=mv^{m-1}t\frac{\partial}{\partial u}+x\frac{\partial}{\partial v}
\]
be the locally nilpotent derivations defining the $\mathbb{G}_{a,k}$-actions
on $Y$ and $X_{m}$ respectively with algebraic spaces quotients
$Y_{*}/\mathbb{G}_{a,k}\cong\mathfrak{S}_{m}\cong X_{m}/\mathbb{G}_{a,k}$.
By the construction used in the proof of Theorem \ref{thm:MainTheorem},
we have a commutative diagram \[\xymatrix{ & W_m=Y_{*}\times_{\mathfrak{S}_{m}}X_{m} \ar[ddl] \ar[ddr]  \ar[r]^{\cong} & X_m\times \mathbb{A}^1_k \ar[dd] \ar[r]^{\cong} & X_m\times_{\mathbb{A}^2_{x,v}}X_1  \ar[ddl]  \ar[ddr]\ar[r]^{\cong}& \mathrm{SL}_2\times \mathbb{A}^1_k \ar[dd] \\ & \\ Y_* \ar[dr] & & X_m \ar[dl] \ar[ddr] & & X_1=\mathrm{SL_2} \ar[ddl]  \\ & \mathfrak{S}_m \ar[d] \\ & \mathbb{A}^2_{x,t} \ar[dr] & &  \mathbb{A}^2_{x,v} \ar[dl] \\ & & \mathbb{A}^1_{x}=\mathrm{Spec}(k[x]).}\] 

Let $\tilde{\partial}_{m}$ and $\tilde{\partial}$ be the commuting
locally nilpotent $k[x,t]$-derivations of the coordinate ring of
$W_{m}$ with kernels equal to $\Gamma(Y,\mathcal{O}_{Y})$ and $\Gamma(X_{m},\mathcal{O}_{X_{m}})$
corresponding to the $\mathbb{G}_{a,k}$-torsors $\mathrm{pr}_{1}:W_{m}=Y_{*}\times_{\mathfrak{S}_{m}}X_{m}\rightarrow Y_{*}$
and $\mathrm{pr}_{2}:W_{m}=Y_{*}\times_{\mathfrak{S}_{m}}X_{m}\rightarrow X_{m}$
respectively. Via the left-hand side isomorphism $\psi:W_{m}\rightarrow X_{m}\times\mathbb{A}_{k}^{1}$
of the top line of the diagram, the derivation $\tilde{\partial}$
corresponds to the locally nilpotent derivation $\frac{\partial}{\partial w}$
of the coordinate ring $\Gamma(X_{m},\mathcal{O}_{X_{m}})[w]$ of
$X_{m}\times\mathbb{A}_{k}^{1}$, whereas $\tilde{\partial}_{m}$
corresponds to the unique $k[x,t]$-derivation $\delta_{m}$ commuting
with $\frac{\partial}{\partial w}$, whose restriction to $\mathrm{Ker}(\frac{\partial}{\partial w})=\Gamma(X_{m},\mathcal{O}_{X_{m}})$
is equal to $\partial_{m}$ and such that $\delta_{m}(w)=\psi(\tilde{\partial}_{m}(\psi^{-1}(w)).$
In practice, the element $\psi^{-1}(w)\in\mathrm{Ker}\tilde{\partial}$
as well as its image by $\tilde{\partial}_{m}$ can be explicitly
determined by considering an \'etale cover of $S\rightarrow\mathfrak{S}_{m}$
of $\mathfrak{S}_{m}$ on which the $\mathbb{G}_{a,k}^{2}$-torsor
$W_{m}\rightarrow\mathfrak{S}_{m}$ becomes trivial. 

The commutativity of the diagram then implies that the locally nilpotent
derivation $\delta$ of the coordinate ring of $\mathrm{SL}_{2}\times\mathbb{A}_{k}^{1}$
corresponding to $\delta_{m}$ through the isomorphisms 
\[
X_{m}\times\mathbb{A}_{k}^{1}\cong X_{m}\times_{\mathbb{A}_{x,v}^{2}}X_{1}\cong\mathrm{SL}_{2}\times\mathbb{A}_{k}^{1}
\]
of schemes over $\mathrm{Spec}(k[x,v])$ is a $k[x]$-derivation.
Explicit isomorphisms $X_{m}\times\mathbb{A}_{k}^{1}\cong\mathrm{SL}_{2}\times\mathbb{A}_{k}^{1}$
for each $m\geq2$ can be constructed by finding explicit trivializations
of the $\mathbb{G}_{a,k}$-bundles $\mathrm{pr}_{1}:X_{m}\times_{\mathbb{A}_{x,v}^{2}}X_{1}\rightarrow X_{m}$
and $\mathrm{pr}_{2}:X_{m}\times_{\mathbb{A}_{x,v}^{2}}X_{1}\rightarrow X_{1}$.
In practice, this amounts to describe these bundles in terms \v{C}ech
1-cocyles with values in $\mathcal{O}_{X_{m}}$ and $\mathcal{O}_{X_{1}}$
on suitable open cover of $X_{m}$ and $X_{1}$ respectively, and
find explicit expressions of each of these 1-cocyles as coboundaries.
\\

Note in addition that letting $X(m,n,r)=\left\{ v^{m}t^{r}-x^{n}u=1\right\} $,
where $\gcd(m,r)=1$, endowed with the $\mathbb{G}_{a,k}$-action
determined by the locally nilpotent $k[x,t]$-derivation 
\[
\partial_{(m,n,r)}=mv^{m-1}t^{r}\frac{\partial}{\partial u}+x^{n}\frac{\partial}{\partial v},
\]
we have by Proposition \ref{prop:Xmnr-ALgQuotient} isomorphisms 
\[
X_{m}\times\mathbb{A}_{k}^{1}\cong X_{m}\times_{\mathfrak{S}_{m}}X(m,n,r)\cong X(m,n,r)\times\mathbb{A}_{k}^{1}
\]
of schemes over $\mathbb{A}_{x,t}^{2}$. The derivation $\tilde{\partial}_{m}$
corresponds via these isomorphisms to a unique $k[x,t]$-derivation
$\delta_{(m,n,r)}$ of the coordinate ring $\Gamma(X(m,n,r),\mathcal{O}_{X(m,n,r)})[\omega]$
of $X(m,n,r)\times\mathbb{A}_{k}^{1}$ commuting with $\frac{\partial}{\partial\omega}$
and whose restriction to $\mathrm{Ker}(\frac{\partial}{\partial\omega})=\Gamma(X(m,n,r),\mathcal{O}_{X(m,n,r)})$
is equal to $\partial_{(m,n,r)}$. Depending on the deformed Koras-Russell
threefold $Y$, the element $\delta_{(m,n,r)}(\omega)\in\Gamma(X(m,n,r),\mathcal{O}_{X(m,n,r)})$
can be easier to determine for suitably chosen $n$ and $r$ than
the element $\delta_{m}(w)\in\Gamma(X_{m},\mathcal{O}_{X_{m}})$.
For instance, we have: 
\begin{example}
The fixed point free $\mathbb{G}_{a,k}$-actions on 
\[
X_{(2,2,3)}\times\mathbb{A}_{k}^{1}=\mathrm{Spec}(k[x,t,u,v][\omega]/(v^{2}t^{3}-x^{2}u-1))
\]
whose algebraic quotients are the deformed Russell cubic threefolds
$Y_{\alpha}\subset\mathbb{A}_{k}^{4}$ with equations 
\[
x^{2}z=y^{2}-t^{3}+x(1+\alpha t),\quad\alpha\in k,
\]
are given by the locally nilpotent $k[x,t]$-derivations 
\[
\delta_{(2,2,3),\alpha}=2vt^{3}\frac{\partial}{\partial u}+x^{2}\frac{\partial}{\partial v}+(\frac{1}{2}(1+\alpha t)x-t^{3})\frac{\partial}{\partial\omega}.
\]
\end{example}

\bibliographystyle{amsplain}

\end{document}